\documentclass[11pt,letterpaper]{article}
\usepackage[margin=1in]{geometry}

\usepackage{diagbox}
\usepackage{mathtools}
\usepackage{bbm}
\usepackage{latexsym}
\usepackage{epsfig}
\usepackage{amsmath,amsthm,amssymb,enumerate,hyperref}
\usepackage[active]{srcltx}
\usepackage[usenames,dvipsnames]{color}

\def\e{\varepsilon}    
  
 \def\th{\theta}    
 \def\m{\mu}  \def\p{\pi}

 \def\om{\omega}

\def\cD{{\cal D}}
\def\cP{{\cal P}}
\def\cQ{{\cal Q}}

\def\cC{{\cal C}}

\theoremstyle{plain}
\newtheorem{theorem}{Theorem}
\newtheorem{lemma}[theorem]{Lemma}

\newtheorem{conjecture}[theorem]{Conjecture}
\theoremstyle{definition}

\newtheorem{claim}{Claim}
\newtheorem{ass}{Assumption}

\theoremstyle{remark}

\def\cX{{\mathcal X}}

\newcommand{\brac}[1]{\left(#1\right)}

\newcommand{\bfrac}[2]{\left(\frac{#1}{#2}\right)}

\newcommand{\set}[1]{\left\{#1\right\}}

\def\E{\mathbb{E}}

\def\Pr{\mathbb{P}}

\newcommand{\ignore}[1]{}

\newcommand{\beq}[2]{\begin{equation}\label{#1}#2\end{equation}}

\def\cG{\mathcal{G}}
\def\cQ{\mathcal{Q}}

\newcommand{\mults}[1]{\begin{multline*}#1\end{multline*}}

\usepackage{tikz}
\usetikzlibrary{arrows}
\usetikzlibrary{decorations}
\usetikzlibrary{shapes.misc}

\def\cU{{\mathcal U}}

\newcommand{\real}{\ensuremath {\mathbb R} }	
\newcommand{\ent}{\ensuremath {\mathbb Z} }

\newcommand{\tor}{\ensuremath {\mathbb T} }

\newcommand{\remove}[1] {}
\newcommand{\ex} {\E}

\DeclareMathOperator{\Bin}{Bin}
\newcommand{\END}{\text{END}}

\newcommand{\xt}[1]{#1}

\begin{document}
\author{Alan Frieze\thanks{Department of Mathematical Sciences, Carnegie Mellon University, Pittsburgh PA, USA, 15213. Research supported in part by NSF grant DMS1661063.} and Xavier P\'erez-Gim\'enez\thanks{Department of Mathematics, University of Nebraska-Lincoln, Lincoln NE, USA, 68588. Research supported in part by Simons Foundation Grant \#587019 and by NSF grant DMS2201590.}}

\title{Rainbow Hamilton Cycles in Random Geometric Graphs}
\maketitle

\begin{abstract}
Let $X_1,X_2,\ldots,X_n$ be chosen independently and uniformly at random from the unit $d$-dimensional cube $[0,1]^d$. Let $r$ be given and let $\cX=\set{X_1,X_2,\ldots,X_n}$. The random geometric graph $G=G_{\cX,r}$ has vertex set $\cX$ and an edge $X_iX_j$ whenever $\|X_i-X_j\|\leq r$. We show that if each edge of $G$ is colored independently from one of $n+o(n)$ colors and $r$ has the smallest value such that $G$ has minimum degree at least two, then $G$ contains a rainbow Hamilton cycle a.a.s.
\end{abstract}

\section{Introduction}
Given a graph $G=(V,E)$ plus an edge coloring $c:E\to[q]$, we say that $S\subseteq E$ is {\em rainbow colored} if no two edges of $S$ have the same color. There has been a substantial amount of research on the question as to when \xt{an edge colored graph contains} a rainbow Hamilton cycle. The early research was done in the context of the complete graph $K_n$ when restrictions were placed on the colorings. In this paper we deal with the case where we have a random geometric graph and the edges are colored randomly.

In the case of the Erd\H{o}s-R\'enyi random graph $G_{n,m}$, Cooper and Frieze~\cite{CF2} proved that if $m\geq 21n\log n$ and each edge of $G_{n,m}$ is randomly given one of at least $q\geq 21n$ random colors then {\em asymptotically almost surely} (a.a.s.) there is a rainbow Hamilton cycle. Frieze and Loh~\cite{FLo} improved this result to show that if $m\geq \frac{1}{2}(n+o(n))\log n$ and $q\geq (1+o(1))n$ then a.a.s.\ there is a rainbow Hamilton cycle. This was further improved by Ferber and Krivelevich~\cite{FeKr} to $m=n(\log n+\log\log n+\om)/2$ and $q\geq (1+o(1))n$, where $\om\to\infty$ with $n$. This is best possible in terms of the number of edges. The case $q=n$ was considered by Bal and Frieze~\cite{BalF}. They showed that $O(n\log n)$ random edges suffice.

Let $X_1,X_2,\ldots,X_n$ be chosen independently and uniformly at random from the unit $d$-dimensional cube $[0,1]^d$ where $d\geq 2$ is constant. Let $r$ be given and let $\cX=\set{X_1,X_2,\ldots,X_n}$.
The random geometric graph $G_{\cX,r}$ has vertex set $[n]$ and an edge $ij$ for each pair $i,j\in[n]$ ($i\ne j$) satisfying $\|X_i-X_j\|\leq r$. Here $\|\cdot\|$ refers to an arbitrary $\ell_p$-norm, where $1<p\leq\infty$.
We define the {\em length} of an edge $ij$ to be $\|X_i-X_j\|$.
Throughout the paper we tacitly assume that the points $X_1,\ldots,X_n$ are all different, which happens almost surely, and identify the vertex set with $\cX$. (We will use the terms point and vertex interchangeably when referring to an element of $\cX$.)
Suppose now that each edge of $G_{\cX,r}$ is given a random color from $[q]$. We call the resulting edge-colored graph $G_{\cX,r,q}$. Bal, Bennett, P\'erez-Gim\'enez and Pralat~\cite{BBPP} considered the problem of the existence of a rainbow Hamilton cycle \xt{in} $G_{\cX,r,q}$. They showed that for $r$ at the threshold for Hamiltonicity, $q=O(n)$ random colors are sufficient to have a rainbow Hamilton cycle a.a.s. The aim of this paper is to show that $q=n+o(n)$ colors suffice in this context.

Let $\th=\th(d,p)$ denote the volume of the unit $\ell_p$-ball in $d$ dimensions, and let
\beq{defr}{
r^d=\frac{(2/d)\log n+(4-d-2/d)\log\log n+f}{2^{2-d}\th n},
}
for some $f=f(n)$.
\begin{theorem}\label{th1}
Let $r$ be as in \eqref{defr} for some
$f\to\infty$. Let $\eta>0$ be an arbitrarily small constant and $q=\lceil(1+\eta)n\rceil$. Then a.a.s.\ $G_{\cX,r,q}$ contains a rainbow Hamilton cycle.
\end{theorem}
We actually prove a stronger hitting-time result, for which we need some definitions. For $n\ge3$, let 
\[
\hat{r} = \inf\set{r\geq 0:\;G_{\cX,r}\text{ has minimum degree at least 2}}.
\]
Clearly, $\hat r$ is a deterministic continuous function of the random set of points $\cX$ and thus a random variable.
The random graph $G_{\cX,\hat{r}}$ can be obtained by taking an empty graph on vertex set $\cX$ and adding edges one by one in increasing order of lengths until the minimum degree becomes $2$ or more. (If two or more edges have the same length, they should be added all at once to the graph, but this does not happen almost surely.) In particular, $G_{\cX,\hat{r}}$ has minimum degree at least $2$, so the infimum in the definition of $\hat r$ can be safely replaced by a minimum.
The asymptotic distribution of $\hat r$ is well known, and can be derived from Theorem~8.4 in~\cite{Penrose}. Indeed, with $r$ parametrized in terms of $f$ as in~\eqref{defr}, we have
\beq{Pen}{
\lim_{n\to\infty}\Pr(\hat r \le r) = \begin{cases}
0 & f\to-\infty
\\
F(\alpha) & f\to \alpha\in\real
\\
1 & f\to\infty,
\end{cases}
}
where $F(\alpha)$ is a continuous distribution function. (An explicit description of $F(\alpha)$ can be found, e.g.,  in~Corollaries~3 and~4 of~\cite{BBPP}.)
\xt{Also well known are the facts that a.a.s.\ $G_{\cX,\hat{r}}$ is $2$-connected~\cite{Penrose} and contains a Hamilton cycle~\cite{BBKMW,MPW}. Motivated by all the above, w}e now consider the edge-colored version $G_{\cX,\hat{r},q}$ of $G_{\cX,\hat{r}}$.
Our main result asserts that, if we start with the empty graph on vertex set $\cX$ and we add randomly colored edges one by one in increasing order of lengths, then (provided that we use sufficiently many colors) a.a.s.\ we obtain a rainbow Hamilton cycle as soon as the minimum degree becomes at least $2$.
\begin{theorem}\label{th2}
Let $\eta>0$ be any fixed constant and $q=\lceil(1+\eta)n\rceil$. Then $G_{\cX,\hat r,q}$ has a rainbow Hamilton cycle a.a.s.
\end{theorem}
Combining this and~\eqref{Pen} immediately yields Theorem~\ref{th1}, so we will devote the remainder of the paper to the proof of Theorem~\ref{th2}.


\paragraph{Proof sketch.} We partition $[0,1]^d$ into small cubic cells of side around $\e r$\xt{, where $\e>0$ is constant that is chosen to be sufficiently small given $\eta$ and $d$, and $r$ is a deterministic function of $n$ which is a.a.s.\ slightly below $\hat r$}. These cells are classified into types acording to the number of points and color repetitions they contain.
Then the set of cells is endowed with a graph structure by connecting every pair of cells at distance slightly less than $r$. (Note that similar constructions have been fruitfully used in~\cite{BBPP,BBKMW,DMP,MPW}.)
In Section~\ref{sec:struc}, we derive some basic properties of this graph of cells.
Then we use a variation of P\'osa's rotation-extension argument to show that most cells contain a spanning family of `not too many' rainbow paths that avoid certain forbidden colors. This type of argument has been widely applied in the study of Hamilton cycles in many other families of random graphs (e.g. the Erd\H{o}s-R\'enyi random graph $G_{n,m}$ \cite{KS},  random regular graphs \cite{FF}, preferential attachment graphs~\cite{FPPR}), but so far not before in the context of random geometric graphs.
In Section~\ref{sec:coloring}, we introduce and analyze a greedy procedure ({\tt Build}), which a.a.s.~constructs a rainbow Hamilton cycle in $G_{\cX,\hat r,q}$, based on the structure and properties of the graph of cells.
An unusual and interesting feature of this procedure is that it sometimes introduces errors (i.e.~color repetitions) which are recursively fixed by another procedure ({\tt Problem-fix}), which may in turn trigger further errors. We show that typically these errors do not accumulate past a certain bound and the algorithm succeeds.

\paragraph{Note about parameters $d$ and $p$.}
Recall that both Theorems~\ref{th1} and~\ref{th2} assume $d\ge 2$ and $1<p\le \infty$. The former assumption is not superfluous as the $1$-dimensional case is significantly different. Indeed, when $d=1$, vertices of degree less than $2$ are no longer the main obstruction to the existence of Hamilton cycles (rainbow or not). In fact, even for $r$ well above the sharp threshold $\hat r$ for the minimum degree being at least $2$, one will typically find many empty ``gaps'' in $[0,1]$ of length greater than $r$ between pairs of consecutive vertices of $\cX$. These gaps prevent $G_{\cX,\hat{r}}$ from being connected and thus from having a Hamilton cycle.
On the other hand, our results may still be true for $p=1$. The only reason why we exclude the $\ell_1$-norm case is because that is required in some of the technical lemmas from earlier papers. More precisely, Lemma~\ref{lem1} (which is proved in~\cite{BBPP,MPW}) relies on a result by Penrose (Thm 13.17 in~\cite{Penrose}), which asserts that a.a.s.\ $G_{\cX,\hat{r}}$ is $2$-connected. Penrose's result assumes $1<p\le\infty$ for technical reasons in the argument, but it is plausible that it still holds for $p=1$, in which case our results could be extended as well.

\paragraph{\xt{Further remarks and open problems.}}
We finish the discussion by observing that our results are best possible in terms of the number of permitted colors $q$.
Indeed, for dimension $d\in\{2,3\}$, if we allow only $q=n$ colors, then a standard coupon collector argument shows that  a.a.s.\ some colors are still missing on $G_{\cX,\hat r,q}$.
More precisely, let
\[
r^* = \inf\set{r\geq 0:\; \text{all $n$ colors appear on $G_{\cX,r,n}$}}.
\]
Then a.a.s.\ $r^*\sim \sqrt[d]{2\log n / (\theta n)}$, and thus
\[
\begin{cases}
r^* \ge (3/2+o(1))\hat r & \text{for $d\in\{2,3\}$}
\\
r^* \sim \hat r & \text{for $d=4$}
\\
r^*\le (5/8+o(1))\hat r & \text{for $d\ge 5$}.
\end{cases}
\]
Similarly, if we consider a slight variation of the model in which the points of $\cX$ are placed on the torus $\tor^d := \real^d/\ent^d$ instead of the cube $[0,1]^d$, then with the analogous definitions of $G_{\cX,\hat r,q}$, $\hat r$ and $r^*$, we have that a.a.s.~$\hat r\sim \sqrt[d]{\log n / (\theta n)}$ (see~Theorem~8.3 in~\cite{Penrose}) and therefore $r^*\sim \sqrt[d]{2} \cdot \hat r$.
The difference in $\hat r$ between the two models is explained by the presence of vertices of degree less than $2$ near the boundaries of $[0,1]^d$.
In either case, it is conceivable that with exactly $q=n$ colors, as soon as the minimum degree is at least $2$ and we see all the colors, we have a rainbow Hamilton cycle a.a.s.
We state this as a conjecture for either the cube $[0,1]^d$ or the torus $\tor^d$ models. Let $t_1\vee t_2:=\max\{t_1,t_2\}$.
\begin{conjecture}
$G_{\cX, \hat r \vee r^*,n}$ has a rainbow Hamilton cycle a.a.s.
\end{conjecture}
We also include a similar statement conditional on the event that $G_{\cX, \hat r,n}$ has all $n$ colors (which is a rare event for the cube model and $d\in\{2,3\}$ or for the torus model and any $d\ge2$).
\begin{conjecture}
Conditional upon $r^*\le \hat r$, $G_{\cX, \hat r,n}$ has a rainbow Hamilton cycle a.a.s.
\end{conjecture}
While this paper only discusses rainbow Hamilton cycles, analogous questions can be asked about rainbow perfect matchings with $q=n/2$ colors.
Let
\[
\hat r_1 = \inf\set{r\geq 0:\;G_{\cX,r}\text{ has minimum degree at least 1}}.
\]
and (for even $n$)
\[
r^*_1 = \inf\set{r\geq 0:\; \text{all $n/2$ colors appear on $G_{\cX,r,n/2}$}}.
\]
\begin{conjecture}
For even $n$, $G_{\cX, \hat r_1 \vee r^*_1,n/2}$ has a rainbow perfect matching a.a.s.
\end{conjecture}
\begin{conjecture}
For even $n$ and conditional upon $r^*_1\le \hat r_1$, $G_{\cX, \hat r_1,n/2}$ has a rainbow perfect matching a.a.s.
\end{conjecture}


\section{Notation and structural properties}\label{sec:struc}
%
%
Throughout the paper, $d\ge2$ and an $\ell_p$-norm $\|\cdot\|$ on $\real^d$ ($1<p\le\infty$) are fixed.
Let $\eta>0$ be an arbitrary constant (which we will assume to be sufficiently small to satisfy all the requirements in the argument) and set 
\[
q = \lceil(1+\eta) n\rceil.
\]
Let $Q=[q]$ denote the set of available colors.
Recall that $G_{\cX,\hat r,q}$ is obtained by assigning to each edge of $G_{\cX,\hat r}$ a random color in $Q$ chosen uniformly at random and independently from all other choices.


Let $\e>0$ be a constant which is assumed to be sufficiently small given our choices of $\eta$ and $d$.
We use the standard $o()$, $\omega()$, $O()$, $\Theta()$ and $\Omega()$ asymptotic notation as $n\to\infty$ with the following extra considerations. We do not assume any sign on a sequence $a_n$ satisfying $a_n=o(1)$ or  $a_n=O(1)$, but on the other hand a sequence satisfying $a_n=\Theta(1)$, $a_n=\Omega(1)$ or $a_n=\omega(1)$ is assumed to be positive for all but finitely many $n$. Furthermore, the constants involved in the bounds of the definitions of $O()$, $\Theta()$ and $\Omega()$ may depend on $d$ as well as some other parameters, but not on $\eta$ or $\e$. Whenever these constants depend on our choice of $\e$ (in addition to $d$ or other parameters), we use the alternative notation $O_\e()$, $\Theta_\e()$ and $\Omega_\e()$ instead.


Henceforth, let $r$ be defined as in~\eqref{defr} for some arbitrary function $f\to-\infty$, $f=o(\log\log n)$.
\xt{(The reason why we take $f=o(\log\log n)$ is that we plan to mimic some of the definitions in~\cite{BBPP} where this assumption is made, but taking $f=o(\log n)$ would also work.)}
From~\eqref{Pen}, we have
\[
r\le \hat r
\qquad\text{and}\qquad
r\sim\hat r
\qquad\text{a.a.s.}
\]
Most of the (colored) edges that we will consider in our argument have length at most $r$, but we will need a few longer edges of length up to $\hat r$ to be able to close the rainbow Hamilton cycle.
The main advantage of working with parameter $r$ instead of $\hat r$ is that the former is deterministic, whereas the latter is random.
Following the construction in~\cite{BBPP}, we divide $[0,1]^d$ into a set $\cC$ of $N=\lceil(\e r)^{-1}\rceil^d$ $d$-dimensional cubic cells of side $s = 1/\lceil(\e r)^{-1}\rceil \sim\e r$. We remark that
\[
N \sim \frac{d \th n}{2^{d-1}\e^d\log n}.
\]
For sake of simplicity, assume that every point in $\cX$ is contained in one single cell in $\cC$ (which occurs almost surely, since cell boundaries have measure $0$).
The  {\em graph of cells} $\cG_\cC$ is a graph with vertex set $\cC$ where two cells are adjacent in $\cG_\cC$ if their centres are at $\ell_p$-distance at most $r - \xt{ds}$. (Here we assume that $\xt{ds}$ is much smaller than $r$ by our choice of $\e$.)
By the triangle inequality, any two different points  $X_i,X_j\in\cX$ which are contained in the same cell or in two cells that are adjacent in $\cG_\cC$ satisfy
$\|X_i-X_j\| \le r$,
and therefore $X_iX_j$ must be an edge of $G_{\cX,r}$ and a.a.s.\ an edge of $G_{\cX,\hat r}$.
\xt{In other words, the vertices contained in one cell or in two adjacent cells in $\cG_\cC$ induce a clique in $G_{\cX,r}$.
Moreover, note that the set of cells adjacent in $\cG_\cC$ to a given cell is contained in an $\ell_\infty$-ball of radius $r$ which has volume $(2r)^d$. Then, since each cell has volume $(1+o(1))(\e r)^d$, we conclude that
\beq{eq:degree}{
\text{the graph of cells $\cG_\cC$ has maximum degree $O_\e(1)$.}
}

}
A cell $C$ is {\em dense} if $|C\cap\cX|\geq \e^3\log n$. Otherwise it is {\em sparse}. The set of dense cells is denoted by $\cD$, and $\cG_\cD$ is the subgraph of $\cG_\cC$ induced by the \xt{dense} cells. The paper~\cite{BBPP} shows that a.a.s. 
\xt{
\beq{P1}{
\text{the largest component $\Gamma_0$ of $\cG_\cD$ contains $N-o(N)$ cells.}
}
As it is customary in the field, we call $\Gamma_0$ the {\em giant} component of $\cG_\cD$.
}
(The proof in~\cite{BBPP} is adapted from an earlier article~\cite{MPW} that uses a less restrictive definition of dense cell.)
The cells in $\Gamma_0$ are called {\em good}. A cell that is not good, but is adjacent \xt{(in $\cG_\cC$)} to a cell in $\Gamma_0$ is called {\em bad}. The remaining cells are called {\em ugly}. Note that bad cells are sparse by definition, but ugly cells may be dense or sparse. The following two lemmas describe properties that occur a.a.s., and their proofs are in~\cite{BBPP}: 
\begin{lemma}\label{lem0}A.a.s.
\begin{enumerate}[{\bf P1}]
\item $|C\cap\cX|\leq \log n$ for all $C\in\cC$ (cf.~Lemma~5 in~\cite{BBPP}).
\item There are at most $n^{1-\e/2}$ bad cells (cf.~Lemma~10 in~\cite{BBPP}).
\item There are at most $n^{O(\e^{1/d})}$ ugly cells (cf.~Lemma~10 in~\cite{BBPP}).
\item The maximum degree in \xt{$G_{\cX,\hat r}$} is at most $O(\log n)$ (cf.~Lemma~6 in~\cite{BBPP}).
\end{enumerate}
\end{lemma}
\begin{lemma}[Lemma~13 in~\cite{BBPP}]\label{lem1}
Let $\cX_\cU$ denote the set of points in ugly cells\xt{, and let $A>0$ be an arbitrary constant}. Then a.a.s.\ \xt{$G_{\cX,\hat r}$ has} a collection of paths $\cP$ such that
\begin{enumerate}[{\bf Q1}]
\item $\cP$ covers $\cX_{\cU}$.
\item $\cP$ covers at most two vertices inside \xt{each} non-ugly cell.
\item Every vertex in $\cX$ that is covered by $\cP$ is at graph-distance at most $2(20d)^d$ from some vertex in $X_\cU$ with respect to the graph
$G_{\cX,\hat r}$.
\item For each path $P\in\cP$, there is a good cell $C_P$ such that the two endvertices of $P$ lie in cells that are adjacent in $\cG_\cC$ to $C_P$;
\item Every pair of distinct paths in $\cP$ are at $\ell_p$-distance at least $Ar$ from each other.
\end{enumerate}
\end{lemma}
For a region $S \subseteq [0,1]^d$, we let $V(S)=S\cap\cX$ and $E(S)=\binom{V(S)}{2} \cap E(G_{\cX,\hat r})$\xt{, i.e.~$V(S)$ and $E(S)$ are the set of vertices and edges of $G_{\cX,\hat r}$ that are contained inside of the region $S$.}
(Note that this definition will be slightly modified later in~\eqref{eq:VC}\xt{, where a few vertices and edges belonging to some special paths will be removed}.)
\xt{We will typically use the notation $V(S)$ and $E(S)$ when $S$ is a cell or a union of cells. Recall that any two points in a cell $C$ are at $\ell_p$-distance much smaller than $r$ or $\hat r$, and thus $V(C)$ induces a clique in $G_{\cX,\hat r}$.}
A color repetition in $S$ is a pair of edges in $E(S)$ that receive the same color in $G_{\cX,\hat r,q}$.
A cell $C$ is {\em rainbow} if $E(C)$ is rainbow: that is, $C$ has no color repetitions.

We now prove some lemmas related to the colorings of cells.
\xt{Hereafter, the $\ell_p$-distance between two cells or between a cell and the boundary of $[0,1]^d$ is measured from the center of the cell(s).}
\begin{lemma}\label{lemcols}
For any constant $A>0$, the following hold  a.a.s.
\begin{enumerate}[(a)]
\item
There are at most $\log^4n$ non-rainbow cells.
\item
No $d$-dimensional cube of side at most $Ar$ obtained as a union of cells in $\cC$ contains $2$ color repetitions.
\item
There are no two non-rainbow cells within $\ell_p$-distance $Ar$ of each other.
\item
There are no non-rainbow cells within $\ell_p$-distance $Ar$ of the boundary of $[0,1]^d$.
\item
There are no non-rainbow cells within $\ell_p$-distance $Ar$ of any cell that is not good.
\end{enumerate}
\end{lemma}
\paragraph{Remark.}
In particular (with $A>1$), a.a.s.\ every non-rainbow cell is good and is only adjacent in $\cG_\cC$ to good rainbow cells.
Moreover, a.a.s.\ every cell contains at most one \xt{color} repetition.

\begin{proof}
All the statements in the lemma follow from simple first moment arguments. \xt{We will also make repeated use of the following simple fact, which follows immediately from Markov's inequality. Given a random variable $Y$ with $\Bin(m,t)$ distribution and any integer $0\le k\le m$,
\begin{equation}\label{eq:silly}
\Pr(Y\ge k) = \Pr\left(\binom{Y}{k}\ge 1\right) \le \ex \binom{Y}{k} = \binom{m}{k} t^k \le (mt)^k.
\end{equation}
(In other words, the probability of having at least $k$ successful trials is at most the expected number of sets of $k$ successful trials.)
}

(a) For a fixed cell $C$, 
\beq{prain}{
\Pr(C\text{ is not rainbow}\mid \text{\bf P1})\leq \lceil(1+\eta)n\rceil \Pr\brac{\Bin\brac{ \lfloor\log^2n\rfloor, \frac{1}{\lceil(1+\eta)n\rceil}}\geq 2}=O\bfrac{\log^4n}{n}.
}
{\em Explanation:} we choose a color $c$. Then the number of edges of color $c$ in cell $C$ is dominated by the stated binomial. \xt{We also use~\eqref{eq:silly} with $k=2$.}

We then have, by the Markov inequality that
\[
\Pr(\neg\ (a)\mid \text{\bf P1})\leq \frac{\ex( \text{number of \xt{non-}rainbow cells} \mid \text{\bf P1})}{\log^4 n} = O\bfrac{N\log^4n}{n\log^4n}=O_\e\bfrac{1}{\log n}.
\]

(b)--(c) Let $\cQ$ be the set of regions $Q\subseteq[0,1]^d$ such that $Q$ is a $d$-dimensional cube of side at most $(A+1)r$ obtained as a union of cells in $\cC$. Note that $|\cQ|=O_\e(N)=O_\e(n/\log n)$. Moreover, assuming~{\bf P1}, $|V(Q)|=O_\e(\log n)$ and thus $|E(Q)|=O_\e(\log^2n)$ for each $Q\in\cQ$. Therefore, we have that
\begin{align}
\Pr(\text{some $Q\in\cQ$ has $2$ color repetitions}\mid \text{\bf P1})
& \le |\cQ|\lceil(1+\eta)n\rceil\Pr\brac{\Bin\brac{O_\e(\log^2n) ,\frac{1}{\lceil(1+\eta)n\rceil}}\geq 3} + 
\notag\\
&
+ |\cQ|\lceil(1+\eta)n\rceil^2\Pr\brac{\Bin\brac{O_\e(\log^2n) ,\frac{1}{\lceil(1+\eta)n\rceil}}\geq 2}^2
\notag\\
&=O_\e\bfrac{\log^7n}{n}.
\label{eq:Q}
\end{align}
{\em Explanation:}  The first term is an upper bound on the expected number of triples of edges of the same color and the second term accounts for double pairs of edges with the same color. \xt{We also use~\eqref{eq:silly} for each term with $k=3$ and $k=2$, respectively.}

Clearly,~\eqref{eq:Q} implies~(b). It also implies~(c) since any two cells within $\ell_p$-distance $Ar$ must be contained in one $Q\in\cQ$ \xt{(assuming $\e<1/2$)}.

(d)--(e)
Assuming {\bf P2} and {\bf P3}, there are at most $2n^{1-\e/2}$ cells that are not good.
\xt{Also, by a trivial volume argument similar to the one leading to~\eqref{eq:degree}, there are $O_\e(1)$ cells within $\ell_p$-distance $Ar$ of any given cell. Hence,}
there are at most $O_\e(n^{1-\e/2})$ cells within $\ell_p$-distance $Ar$ of some cell that is not good. Moreover, there are $O_\e(1/r^{d-1})$ cells within $\ell_p$-distance $Ar$ of the boundary of $[0,1]^d$.
Arguing as in~(a), we have that
\begin{align*}
\Pr(\neg\ (d) \text{ or } \neg (e) \mid \text{\bf P1,P2,P3})
&= O_\e \left(n^{1-\e/2} + 1/r^{d-1}\right) (1+\eta)n \Pr\brac{\Bin\brac{\log^2n, \frac{1}{(1+\eta)n}}\geq 2}
\\
&= O_\e \left( \frac{\log^4n}{n^{\e/2}} + \frac{\log^4n}{n r^{d-1}} \right)  = o(1).
\qedhere
\end{align*}
\end{proof}

We remove all the non-rainbow cells (which must be good a.a.s.~\xt{by Lemma~\ref{lemcols}(e)}) from the giant component $\Gamma_0$ of good cells, and obtain $\Gamma_1$. We argue next that a.a.s.~$\Gamma_1$ remains connected. \xt{In view of that, we will still refer to $\Gamma_1$ as the giant component.}

\begin{lemma}\label{Kg}
The graph of rainbow good cells $\Gamma_1$ is a.a.s.\ connected.
\end{lemma}
\begin{proof}
Let $C,C'$ be any two cells in $\Gamma_1$ (i.e.~rainbow and good). Since $\Gamma_0$ is connected, there must be a $C,C'$-path $P = (C=C_1,C_2,\ldots,C_m=C')$ of good cells. We want to show that, after deleting all the non-rainbow good cells, there is still a $C,C'$-path in $\Gamma_1$.
Suppose that an interior cell $C_i$ of path $P$ (i.e.~$1<i<m$) is non-rainbow. Let $S$ be the union of all cells different from $C_i$ that are within $\ell_p$-distance $2r$ of $C_i$. Assuming that the a.a.s.\ statements in Lemma~\ref{lemcols} hold, $S$ is away from the boundary of $[0,1]^d$ and all the cells contained in $S$ \xt{(including $C_{i-1}$ and $C_{i+1}$)} are rainbow and good.
Clearly, $S$ is topologically connected and hence the cells in $S$ induce a connected subgraph of $\Gamma_1$.
By construction, $C_{i-1},C_{i+1}\subseteq S$, and hence we can find a $C_{i-1},C_{i+1}$-path $Q$ that uses only cells in $S$ and thus cells that are rainbow and good. Hence, replacing the subpath $C_{i-1},C_i,C_{i+1}$ in $P$ by $Q$, we obtain \xt{a $C,C'$-walk that avoids $C_i$. This walk can be easily turned into a $C,C'$-path, by deleting some cells if needed. Note that by construction the new path avoids $C_i$, and moreover the only new cells that were added are good and rainbow.} Iterating this argument for all non-rainbow cells in $P$, we obtain a $C,C'$-path in $\Gamma_1$.
\end{proof}


We now choose a collection of paths $\cP$ in $G_{\cX,\hat r}$ satisfying {\bf Q1}--{\bf Q5}, which must exist a.a.s.\ in view of Lemma~\ref{lem1}. (If there are multiple choices for $\cP$, pick one arbitrarily.)
We call paths in $\cP$ {\em ugly}.
Let $V(\cP)$ be the set of points in $\cX$ covered by ugly paths, and let $E'(\cP)$ be the set of edges of $G_{\cX,\hat r}$ that are incident with some vertex in $V(\cP)$.
\begin{lemma}\label{rainP}
A.a.s.\ $E'(\cP)$ has $n^{O(\e^{1/d})}$ edges, and it is rainbow colored in $G_{\cX,\hat r,q}$.
\end{lemma}
\begin{proof}
Properties {\bf P1}, {\bf P3}, {\bf P4} and {\bf Q3} immediately imply that $|E'(\cP)| = n^{O(\e^{1/d})}$. Conditional upon this,
the probability that $E'(\cP)$ has a color repetition can be bounded by $\binom{n^{O(\e^{1/d})}}{2}\times n^{-1}=o(1)$. 
\end{proof}
%
In the sequel, we remove all vertices in $V(\cP)$ from the cells, without changing the original cell classification into good, bad and ugly.
(The argument will first attempt to build a rainbow cycle $H$ through $\cX\setminus V(\cP)$ and then insert the paths in $\cP$ into $H$.)
After this operation, {\bf Q1} and {\bf Q2} imply that ugly cells will no longer contain any points from $\cX$ (since they were all on ugly paths and got removed), while each good cell will contain at least $\e^3\log n-2$ points from $\cX$ (since at most $2$ points were removed).
Note that a.a.s.\  non-rainbow cells are not affected by this operation, since they do not contain points in $V(\cP)$ by Lemma~\ref{lemcols}(e) (with $A>2(20d)^d$) and {\bf Q3}.
For convenience, for each cell $C$, we redefine $V(C)$ and $E(C)$ to denote the sets of vertices and edges of $G_{\cX,\hat r} - V(\cP)$ contained in $C$. That is,
\begin{equation}
V(C) = C\cap\cX\setminus V(\cP)
\qquad\text{and}\qquad
E(C) = \binom{V(C)}{2}.
\label{eq:VC}
\end{equation}
Moreover, let $E'(C)$ be the set of all edges of $G_{\cX,\hat r} - V(\cP)$ incident with some point in $V(C)$.
Finally, we consider the set
$E'$ of all edges of $G_{\cX,\hat r}$ that are incident with points in $V(\cP)$ or with points in cells that are not good or not rainbow. That is,
\begin{equation}
E' = E'(\cP) \cup \bigcup_{C\notin \xt{V(\Gamma_1)} } E'(C).
\label{eq:Eprime}
\end{equation}
During the construction of the rainbow Hamilton cycle in Section~\ref{sec:coloring}, special care will be required to avoid repeating colors that already appear in $E'$. The following result will help us achieve that.
%
%
\begin{lemma}\label{k0}
Let $k_0 = \lceil 20/\e\rceil$.
A.a.s.\ $|E'| \le n^{1-\e/3}$, and moreover,
for every bad or non-rainbow cell $C$, fewer than $k_0 + 2$ edges in $E'(C)$ are assigned a color in $G_{\cX,\hat r,q}$ that is repeated on another edge in $E'$.
\end{lemma}
In fact, we prove something slightly stronger by allowing $C$ to range over all cells, not necessarily bad or non-rainbow. The reason for stating the lemma only for bad or non-rainbow cells is that when we use it in Section~\ref{sec:coloring} we will only expose the colors of edges in $E'$ and assume that the second a.a.s.\ conclusion of the lemma holds as stated just for these cells.
\begin{proof}
Properties {\bf P1}--{\bf P4} and the a.a.s.\ claims in Lemma~\ref{lemcols}(a) and Lemma~\ref{rainP} imply that (eventually, for large $n$)
\[
|E'| \le n^{O(\e^{1/d})} +  O(\log^6n) +  O(n^{1-\e/2} \log^2n) \le n^{1-\e/3}.
\]
Also, for every cell $C$, $|E'(C)| = O(\log^2n)$.
Conditional on all the above properties, the probability that there is a cell $C$ with $k_0$ edges in $E'(C)$ whose colors in $G_{\cX,\hat r,q}$ are also used on $E' \setminus E'(C)$ can be bounded by
\[
N\binom{O(\log^2n)}{k_0}\bfrac{n^{1-\e/3}}{(1+\eta)n}^{k_0} \leq n^{1+o(1)-k_0\e/10} =o(1).
\]
As a result, a.a.s.\ every cell $C$ has fewer than $k_0$ edges in $E'(C)$ with colors repeated on $E' \setminus E'(C)$.
To finish the proof, we observe that, in view of Lemma~\ref{lemcols}(b) (with say $A=1$), a.a.s.\ for every cell $C$ the set of edges $E'(C)$ contains at most one pair of edges with repeated colors.
\end{proof}
Let $G_{m,p}$ denote the Erd\H{o}s-R\'enyi-Gilbert binomial random graph on $m$ vertices where each pair of vertices is joined by an edge with probability $p\in[0,1]$. (Here $p$ is unrelated to the parameter associated to the $\ell_p$-norm $\|\cdot\|$ in the definition of the random geometric graph $G_{\cX,\hat r}$.)
We now prove a lemma concerning the existence of Hamilton cycles and spanning collections of paths in $G_{m,p}$.
\begin{lemma}\label{ham}
\xt{Let $p=p(m)\in[0,1]$ with $p=\Omega(1)$ as $m\to\infty$.}
Then, 
\begin{enumerate}[(a)]
\item $\Pr(G_{m,p}\text{ is not Hamiltonian})\leq e^{-mp/5}$ for $m$ sufficiently large.
\item Let $\psi(G)$ denote the minimum number of vertex disjoint paths that cover the vertices of $G$. Then, for fixed $k\ge 1$ and sufficiently large $m$,
\[
\Pr(\psi(G_{m,p}) > k)\leq e^{-kmp/6}.
\]
\end{enumerate}
\end{lemma}
\begin{proof} (The asymptotic notation in this proof is with respect to $m\to\infty$, and we tacitly assume that $m$ is sufficiently large for every inequality to be true.)

(a) We consider the standard coupling $G_{m,p}\supseteq G_1\cup G_2$, where $G_1,G_2$ are independent copies of $G_{m,p/2}$.
Given a graph $G$ and $S\subseteq V(G)$, let $N_G(S)$ denote the disjoint neighborhood of $S$, i.e.~the set of vertices that are not in $S$ but are adjacent to some vertex in $S$. Let $A_1$ be the event that, for every $S \subseteq V(G_1)$ with $1\le|S|\le m/6$, $|N_{G_1}(S)| > 2|S|$. By bounding the expected number of sets $S$ that violate this condition, we get that
\[
\Pr(\neg A_1)\leq \sum_{s=1}^{\lfloor m/6\rfloor}\binom{m}{s}\binom{m}{2s}(1-p/2)^{s(m-3s)}
\leq  \sum_{s=1}^{\lfloor m/6\rfloor}\brac{\frac{me}{s}\cdot\frac{m^2e^2}{4s^2}\cdot e^{-mp/4}}^s\leq \frac{1}{3} e^{-mp/5}.
\]
Let $A_2$ be the event that $G_1$ is connected. By bounding the expected number of components of order at most $m/2$, we show that
\[
\Pr(\neg A_2)\le \sum_{s=1}^{\lfloor m/2\rfloor} \binom{m}{s} (1-p/2)^{s(m-s)}
\le \sum_{s=1}^{\lfloor m/2\rfloor} \left( \frac{me}{s} \cdot e^{-mp/4} \right)^s \le \frac{1}{3} e^{-mp/5}.
\]
Now let $A_3$ be the event that $G_2$ has at least $\mu := \left\lceil \binom{m}{2}p/2 - m^{7/4} \right\rceil = (1+o(1))m^2p/4$ edges.
By Chernoff's bound (see~e.g.~Corollary~21.7 in~\cite{FK}), $\Pr(\neg A_3) = e^{-\Omega(m^{3/2})}$.

\xt{We will apply P\'osa's rotation-extension argument (see P\'osa~\cite{posa} and also Section~6.2 in~\cite{FK} for more details).
Following the notation in~\cite{FK}, events $A_1$ and $A_2$ imply that, if $G_1$ is not Hamiltonian, then there exists a set $\END \subseteq V(G_1)$ and for each $x\in \END$ a set $\END_x$ with $|\END|,|\END_x| \ge m/6$ with the following property. The addition of any edge $\set{x,y}$ with $x\in\END$ and $y\in\END_x$ (which we call a {\em booster} edge) to $G_1$ results in either increasing the length of the longest path or closing a Hamilton cycle.}
Hence, there must be at least $\binom{\lceil m/6\rceil}{2}$ such boosters.
Moreover, since $A_1,A_2$ are increasing properties with respect to the addition of edges, every non-Hamiltonian supergraph $G_1'\supseteq G_1$ on vertex set $V(G_1)$ must satisfy the same property.
Let us condition on events $A_1$, $A_2$ and $A_3$, and consider an enumeration $e_1,e_2,\ldots,e_\mu,\ldots$ of the edges of $G_2$.
Suppose that, for some $0\le k\le \mu-1$, the supergraph $G_1+\set{e_1,e_2,\ldots,e_k}$ of $G_1$ is not Hamiltonian.
Then the probability that $e_{k+1}$ is a booster is at least $\binom{\lceil m/6\rceil}{2}/\binom{m}{2}\geq 1/37$. (This is because we know that none of $e_1,e_2,\ldots,e_k$ are boosters of $G_1+\set{e_1,e_2,\ldots,e_k}$.) Thus the probability that we fail to produce a Hamilton cycle after adding edges $e_1,\ldots,e_\m$ to $G_1$ is at most $\Pr(\Bin(\m,1/37)\leq m)\leq e^{-\Omega\xt{(m^2)}}$ (again by Chernoff's bound).
Hence, we conclude that
\[
\Pr(G_{m,p}\text{ is not Hamiltonian}) \le \Pr(\neg A_1) + \Pr(\neg A_2) + \Pr(\neg A_3) + e^{-\Omega\xt{(m^2)}} \le e^{-mp/5}.
\]

(b) Let $V_\ell$ be the set of vertices of degree at most $\ell$ in $G_{1}$. Then for $\ell,r=O(1)$,
\mults{
\Pr(|V_\ell|\geq r)\leq \binom{m}{r}\Pr(\Bin(m-r,p/2)\leq \ell)^r\leq m^r\brac{\sum_{i=0}^\ell \binom{m-r}{i}(p/2)^i(1-p/2)^{m-r-i}}^r\\
\leq m^{r+\ell r}e^{-r(m-r-\ell)p/2}\leq \frac{1}{2} e^{-rmp/3}.
}
Suppose now that we arbitrarily add edges incident to the vertices of degree at most $3k$ in $G_1$ so that the new graph $H$ has minimum degree $3k$. (We can follow any fixed deterministic rule to do that, so $H$ is a well-defined function of $G_1$.)
Let $A'_1$ be the event that, for every $S \subseteq V(H)$ with $1\le|S|\le m/6$, $|N_{H}(S)| > 2|S|$, and let $A'_2$ be the event that $H$ is connected. Then, arguing as in~(a),
\begin{align*}
\Pr(\neg A'_1) &\leq \sum_{s=k}^{\lfloor m/6\rfloor}\binom{m}{s}\binom{m}{2s}(1-p/2)^{s(m-3s)}
\leq  \sum_{s=k}^{\lfloor m/6\rfloor}\brac{\frac{me}{s}\cdot\frac{m^2e^2}{4s^2}\cdot e^{-mp/4}}^s\leq \frac{1}{2}e^{-kmp/5}.
\\
\Pr(\neg A'_2) &\le \sum_{s=3k}^{\lfloor m/2\rfloor} \binom{m}{s} (1-p/2)^{s(m-s)}
\le \sum_{s=3k}^{\lfloor m/2\rfloor} \left( \frac{me}{s} \cdot e^{-mp/4} \right)^s \le e^{-3kmp/5}.
\end{align*}
Repeating the same P\'osa rotation-extension argument from part~(a), it then follows that
\[
\Pr(H\cup G_2\text{ is not Hamiltonian})\le \Pr(\neg A'_1) + \Pr(\neg A'_2) + \Pr(\neg A_3) + e^{-\Omega\xt{(m^2)}} \le e^{-kmp/5}.
\]
Now suppose that $k\ge 2$. If $|V_{3k}|\leq \lfloor k/2\rfloor$ and $H\cup G_2$ is Hamiltonian, then we have $\psi(G_{m,p})\le k$, since deleting all the edges in $E(H)\setminus E(G_1)$ from a Hamilton cycle of $H\cup G_2$ creates at most $2\lfloor k/2\rfloor$ paths.
Hence,
\[
\Pr(\psi(G_{m,p})> k)\le e^{-kmp/5} + \Pr(|V_{3k}|\ge \lfloor k/2\rfloor+1) \le e^{-kmp/5} + \frac{1}{2} e^{-kmp/6} \le e^{-kmp/6}.
\]
The case $k=1$ follows immediately from part~(a). This finishes the proof of the lemma.
\end{proof}


\section{Rainbow Hamilton cycle construction}\label{sec:coloring}
We now describe how we select our rainbow Hamilton cycle. 
Firstly, for each point $X_i\in\cX$, we expose the cell containing $X_i$ (which determines which cells are good, bad and ugly), and suppose that properties {\bf P1}--{\bf P3} in Lemma~\ref{lem0} hold. (Note that we do not reveal the exact location of each point $X_i$ in $[0,1]^d$ to avoid conditioning on events of measure $0$.) Next, we expose the incidence structure of graph $G_{\cX,\hat r}$, and suppose {\bf P4} in Lemma~\ref{lem0} also holds. Moreover, assume there is a collection of ugly paths $\cP$ that satisfies {\bf Q1}--{\bf Q5} in Lemma~\ref{lem1}. In the case there is more than one choice for $\cP$, pick one arbitrarily.
Recall $V(\cP)$ is the set of points in $\cX$ covered by paths in $\cP$.
For the next part of the argument we will remove all points in $V(\cP)$ from the cells, and treat them separately (see~\eqref{eq:VC} and the discussion above it, in Section~\ref{sec:struc}). In view of this, for any good cell $C$, $|V(C)|\ge\e^3\log n-2$, while for every ugly cell $D$, $|V(D)|=0$.
Now we reveal the number of color repetitions in each cell (which determines which ones are rainbow) without exposing the actual colors of the edges yet. Assume that all the a.a.s.\ statements in Lemmas~\ref{lemcols} and~\ref{Kg} hold. In particular, every non-rainbow cell must be good, and the graph of rainbow good cells $\Gamma_1$ is connected.
Moreover, points in non-rainbow cells are not adjacent in $G_{\cX,\hat r}$ to points on ugly paths.
Further, we expose the colors of all the edges in $E'$ (defined in~\eqref{eq:Eprime}). Recall that these are the edges of $G_{\cX,\hat r}$ that are incident with points covered by $\cP$ or with points contained in cells that are not in $\Gamma_1$.
Recall the definitions of $E'(\cP)$ and $E'(C)$ in Section~\ref{sec:struc}, as well.
We condition on $E'(\cP)$ being rainbow (which is a.a.s.\ true by Lemma~\ref{rainP}), and suppose that the a.a.s.\ conclusions in Lemma~\ref{k0} hold. In particular $|E'| \le n^{1-\e/3}$.
We conclude this discussion with a crucial observation. Conditional on all the information about $G_{\cX,\hat r, q}$ exposed so far, the colors on the edges $X_iX_j$ with both endpoints in cells of $\Gamma_1$ remain uniformly random with the only restriction that, for every cell $C$ in $\Gamma_1$, the colors on the edges in $E(C)$ must be all different.

In the remainder of this section, we will a.a.s.\ build a rainbow cycle $H$ that visits all the vertices inside rainbow good cells and avoids colors assigned to edges in $E'$. Next, we will deterministically extend $H$ to include all the vertices inside bad or non-rainbow cells. Finally, we will insert the ugly paths into $H$ to create a rainbow Hamilton cycle.


\subsection{Rainbow good cells}\label{ssec:good}
Our first goal is to build a rainbow cycle that covers all the vertices inside rainbow good cells and avoids all the colors used on $E'$.
(Recall $|E'| \le n^{1-\e/3}$.)
Pick a spanning tree $T$ of the giant component $\Gamma_1$ consisting of all the rainbow good cells, and root it at one of its cells $C_1$.
Note that $T$ has maximum degree $\Delta(T)=O_\e(1)$ (\xt{by~\eqref{eq:degree} and since $T$} is a subgraph of $\cG_\cC$), and it contains $N_1$ cells with $N_1\sim N = O_\e(n/\log n)$, in view of all our earlier a.a.s.\ assumptions.
Suppose that $C_1,C_2,\ldots,C_{N_1}$ is an enumeration of the cells in $\Gamma_1$ that follows from a depth-first search of $T$ from the root cell $C_1$.
For each $1<i\le N_1$, let $\pi(i)$ denote the index of the parent $C_{\p(i)}$ of $C_i$ in this search.
For convenience, we write $V_i=V(C_i)$ and $E_i=E(C_i)$.
Let $m_i=|V_i|$, and recall $\e^3\log n-2 \le m_i\le \log n$ from our previous a.a.s.\ assumptions.
Also, for $i,j=1,\ldots,N_1$ ($i\ne j$), let $E_{i,j}$ denote the set of edges in $G_{\cX,\hat r}$ with one endpoint in $V_i$ and one in $V_j$.

Below we describe procedure {\tt Build}, in which we examine the rainbow good cells $C_1,\ldots,C_{N_1}$ in this order and, at each step $i=1,\ldots,N_1$, attempt to construct a rainbow cycle $H_{i} \subseteq G_{\cX,\hat r, q}$ through $V_1\cup\cdots\cup V_i$ that avoids colors on $E'$.
Roughly speaking, at each step $i$, we find either a rainbow cycle or a rainbow collection of paths with vertex set $V_i$ and which does not repeat any colors used on $E'$ or $H_{i-1}$. Then, we patch this cycle or each of these paths into $H_{i-1}$ at the parent cell $C_{\pi(i)}$ by using two edges in $E_{i,\pi(i)}$. This creates the new cycle $H_i$, which is typically rainbow. Occasionally, though, this patching operation cannot be done without repeating some colors already used on $H_{i-1}$. In that case, our algorithm attempts to fix these errors by making a small number of additional modifications to $H_i$, recursively.
In the description of procedure {\tt Build}, it is often convenient to regard $G_{\cX,\hat r}$ as an oriented graph by initially assigning to each edge $\{x,y\}$ an arbitrary orientation, $xy$ or $yx$, which may change over the course of the algorithm.
A path or a cycle is called {\em directed} (with respect to an orientation) if all its vertices have in- and out-degree at most one. We do not assume paths or cycles to be directed unless explicitly stated.

As we run this procedure we will expose some additional information of $G_{\cX,\hat r, q}$, and assume in our description that certain properties hold (see Assumptions~\ref{ass1}--\ref{ass6} below). If any of these assumptions ceases to be true at any given time, then {\tt Build} fails and immediately stops. (We will later show that a.a.s.\ this does not occur.)
Moreover, we claim that some additional properties are satisfied (see Claims~\ref{claim1}--\ref{claim5} below) at the end of each step $i=1,\ldots,N_1$ provided that procedure {\tt Build} has been successful so far. These claims are deterministic consequences of all of our assumptions, and will be proven inductively along with the description of the procedure.

Fix $1\le i\le N_1$, and suppose we have just completed $i$ steps of the algorithm.

\begin{claim}\label{claim1}
$H_i \subseteq G_{\cX,\hat r, q}$ is a rainbow directed cycle on vertex set $V_1\cup\cdots\cup V_i$, and it does not use any colors assigned to $E'$.
\end{claim}

\begin{claim}\label{claim2}
For every $j>i$, the procedure has not yet exposed the colors on any edges in $E_j \cup E_{j,\pi(j)}$. (In particular, these colors remain uniformly distributed conditional upon $E_j$ being rainbow.)
\end{claim}

For convenience, we identify the cycle $H_i$ with its edge set $E(H_i)$, so in particular $|H_i|$ denotes the number of (oriented, colored) edges in $H_i$. We will tacitly follow a similar abuse of notation for other subgraphs of $G_{\cX,\hat r}$ (and also for their corresponding edge-colored versions, given $G_{\cX,\hat r, q}$).
\begin{claim}\label{claim3}
For every $1< j \le i$, $|H_i \cap E_{j,\pi(j)}| = \xt{O_\e(1)}$. Moreover, for every $1\le j'<j$ with $j'\ne\pi(j)$, $|H_i \cap E_{j,j'}|=0$.
\end{claim}

Each of the cells $C_1,\ldots,C_i$ is labelled as {\em safe} or {\em unsafe} (with cell $C_1$ always declared unsafe). Safe cells will be used to fix errors due to color repetitions. Note that some cells may change their status from safe to unsafe during the procedure, but never the other way around.
\begin{claim}\label{claim4}
The number of unsafe cells is at most $o(N_1)$.
\end{claim}

For technical reasons, for each $1\le j\le i$ we select a `reasonably large' matching $M_j$ in $H_i \cap E_j$, and partition it into two disjoint matchings $M'_j$ and $M''_j$ of roughly equal size.
We say that an edge $e$ is incident with a set of edges $A$ in a graph if $e$ shares an endpoint with some edge in $A$.
\begin{claim}\label{claim5}
For every $1\le j \le i$, the following holds. $M_j,M'_j,M''_j \subseteq H_i \cap E_j$ are matchings with $M_j=M'_j\cup M''_j$ and $M'_j\cap M''_j = \emptyset$. These matchings satisfy $|M'_j|, |M''_j| \ge (\e^3/4+o(1)) \log n$ and thus $|M_j| \ge (\e^3/2+o(1)) \log n$. Moreover, if cell $C_j$ is safe, then the procedure has not yet exposed the colors on any edges in $E_{j,\pi(j)}$ that are incident with $M''_{\pi(j)}$.
\end{claim}

\paragraph{Procedure {\tt Build}:}
We initially assign an arbitrary orientation to every edge in $G_{\cX,\hat r}$. First consider cell $C_1$.
We examine the edges in $E_1$ one by one, reveal their color in $G_{\cX,\hat r, q}$, and delete those edges whose color has already been used on $E'$.
(Recall that the colors on $E_1$ are uniformly distributed conditional upon $E_1$ being rainbow.)
Let $G_1$ denote the graph with vertex set $V_1$ and the edges that remain.
Each edge is deleted with probability at most $|E'|/\left(|Q|-\xt{\binom{m_1}{2}}\right) = o(1)$, and thus (ignoring the orientations of the edges) $G_1$ contains a copy of $G_{m_1,p_1}$ with $p_1=1-o(1)$.
\begin{ass}\label{ass1}
$G_1$ is Hamiltonian.
\end{ass}
This holds a.a.s.~by Lemma~\ref{ham}(a). (Recall that if any of our Assumption~\ref{ass1}--\ref{ass6} fails, then {\tt Build} stops and fails.)
Then, pick a Hamilton cycle $H_1$ of $G_1$, which must be rainbow by construction, and modify the orientations of the edges of $H_1$ (if needed) to ensure it is a directed cycle.
Next, select an arbitrary matching $M_1$ of size at least $(\e^3/2+o(1))\log n$ contained in the cycle $H_1$ (e.g.~by taking alternating edges in $H_1$), and partition $M_1$ into two disjoint matchings $M'_1$ and $M''_1$ of size at least $(\e^3/4+o(1))\log n$ each. We label cell $C_1$ as unsafe since we will require all safe cells to have a parent in $T$.
This finalizes the first step of the procedure. Note that Claims~\ref{claim1}--\ref{claim5} are trivially satisfied with $i=1$.

Let $1<i\le N_1$, and suppose we have successfully run the first $i-1$ steps of {\tt Build}. In particular, we inductively assume that Claims~\ref{claim1}--\ref{claim5} were valid at the end of step $i-1$. We now proceed to describe step $i$. As in the first step, we reveal the colors of the edges in $E_i$ one by one, and delete those edges whose color has already been used on $E'$ or $H_{i-1}$. \xt{(As before, recall that the colors on $E_i$ are uniformly distributed conditional upon $E_i$ being rainbow.)} Let $G_i$ denote the resulting graph on vertex set $V_i$.
Each edge is deleted with probability at most
\[
\frac{|E'|+|H_{i-1}|}{|Q|-\xt{\binom{m_i}{2}}} \le \frac{n^{1-\e/3}+n}{(1+\eta)n - \xt{\log^2 n}} = \frac{1+o(1)}{1+\eta} \le 1 -\eta/2 + o(1),
\]
for $\eta<1$. Hence, $G_i$ contains a copy of $G_{m_i,p_i}$ with $p_i=\eta/2+o(1)$. (Here we are again ignoring the current orientations of the edges.)
We say that step $i$ is a Hamiltonian step if $G_i$ contains a Hamilton cycle (i.e.~a cycle through $V_i$, not necessarily directed).
By Lemma~\ref{ham}(a), step $i$ fails to be Hamiltonian with probability at most $e^{-m_ip_i/\xt{5}} \le n^{-\e^3\eta/\xt{10}+o(1)}$.
(This bound is also valid if $i=1$, although a stronger bound was used in the first step of the algorithm.)
\begin{ass}\label{ass2}
The number of non-Hamiltonian steps up to step~$i$ is at most $n^{1-\e^3\eta/\xt{11}}$.
\end{ass}
Note that the expected number of non-Hamiltonian steps at the end of the procedure is at most $N_1 n^{-\e^3\eta/\xt{10}+o(1)} = o(n^{1-\e^3\eta/\xt{11}})$, so Assumption~\ref{ass2} is a.a.s.~valid by the Markov inequality.

If step $i$ is Hamiltonian, we will perform a {\tt \xt{C}ycle-patch} step (below).
\begin{ass}\label{ass3}
If step~$i$ is not Hamiltonian then $G_i$ contains a collection of at most $\psi_0=\lceil\frac{13}{\e^{3}\eta}\rceil$ vertex-disjoint paths that cover $V_i$.
\end{ass}
	Note that, by Lemma~\ref{ham}(b), the probability that Assumption~\ref{ass3} fails at step~$i$ is at most $e^{-\psi_0 m_i p_i/6} \le n^{-\psi_0 \e^3\eta/12+o(1)} = o(1/N_1)$. Taking a union bound over all $N_1$ steps in the algorithm, we conclude that a.a.s.\ Assumption~\ref{ass3} is always valid. In this case we will perform a {\tt \xt{F}orest-patch} step (below). 

\paragraph{Swaps and cycle rotations:}
For the description of the {\tt \xt{C}ycle-patch} and {\tt \xt{F}orest-patch} steps below, it is convenient to introduce the following operations in the context of a directed \xt{graph} where loops are allowed. Given two non-incident directed edges $xy$ and $uv$ (possibly $x=y$ or $u=v$), an $xy,uv$-{\em swap} is the operation that deletes $xy$ and $uv$ and replaces them by $xv$ and $uy$.
Note that the orientation of the edges $xy$ and $uv$ determines the way in which their endpoints get recombined into new edges by the $xy,uv$-swap.
We can use swaps to merge or modify directed cycles. For instance, given two vertex-disjoint directed cycles $O_1,O_2$ with $xy\in O_1$ and $uv\in O_2$, the application of an $xy,uv$-swap to $O_1\cup O_2$ yields one single directed cycle on the same vertex set.
(Note that $O_1$ or $O_2$ could be directed cycles of length $1$, i.e.~loops, or of length $2$, i.e.~pairs of anti-parallel edges.)
Moreover, given a directed cycle $O$ of length at least $4$ and two non-consecutive edges $xy,uv$ in $O$, we can reverse the orientations of all edges along the directed path from $y$ to $v$ in $O$ (so that in particular $uv$ becomes $vu$), then apply an $xy,vu$-swap, and finally reverse the orientation of $vy$ to $yv$. We call this operation an $xy,uv$-{\em rotation} of the directed cycle $O$. The resulting graph is a different directed cycle with the same vertex set as $O$.

\paragraph{{\tt Cycle-patch} step:}

If $G_i$ is Hamiltonian, then label cell $C_i$ as safe and pick a Hamilton cycle $D_i$ of $G_i$. We can assume that $D_i$ is a directed cycle, by appropriately modifying the orientations of the edges if necessary.
Note that $H_{i-1} \cup D_i$ is rainbow by construction and does not use any colors from edges in $E'$.
Since $D_i$ is a cycle with $m_i\ge(\e^3+o(1))\log n$ edges, we can choose a matching $M_i$ of size at least $(\e^3/2+o(1))\log n$ contained in $D_i$. Moreover, recall that $M'_{\pi(i)} \subseteq E_{\pi(i)} \cap H_{i-1}$ is a matching of size at least
$(\e^3/4+o(1))\log n$, by Claim~\ref{claim5} applied to step $i-1$.
Now let us reveal the colors on all the edges in $E_{i,\pi(i)}$ that are incident with both $M_i$ and $M'_{\pi(i)}$. These colors had not been exposed yet in view of Claim~\ref{claim2} at step $i-1$.
Our goal is to merge $H_{i-1}$ and $D_i$ together into one single larger directed cycle, which we will call $H_i$.
To do that, we will pick appropriate edges $xy\in M_i$ and $uv\in M'_{\pi(i)}$, and perform an $xy,uv$-swap to $H_{i-1}\cup D_i$.
That is, edges $xy$ and $uv$ are replaced by $xu$ and $yv$, by appropriately updating edge orientations in $G_{\cX,\hat r}$ if needed. (We could also merge $H_{i-1}$ and $D_i$ in a different way if we first reversed the orientation of the edges in $D_i$ and then applied an $xy,vu$-swap instead, but our argument will ignore this alternative.)
An $xy$,$uv$-swap is {\em valid} if the two added edges, $xu$ and $yv$, receive different colors in $G_{\cX,\hat r, q}$ and these colors have not already been used on $E' \cup H_{i-1} \cup E_i$. Note that, in that case, the cycle $H_i$ resulting from the swap satisfies the properties in Claim~\ref{claim1}.
\begin{ass}\label{ass4}
There are indeed edges $xy\in M_i$ and $uv\in M'_{\pi(i)}$ such that the $xy$,$uv$-swap is valid.
\end{ass}
The probability that a given $xy$,$uv$-swap is valid is at least
\[
\left( 1 - \frac{|E'|+|H_{i-1}|+|E_i|+1}{|Q|} \right)^2 \ge \left( 1 - \frac{1+o(1)}{1+\eta} \right)^2 \ge \eta^2/2,
\]
for $\eta<\sqrt2-1$ and large enough $n$.
Since $M_i$ and $M'_{\pi(i)}$ are disjoint matchings, the pairs of edges added in different swaps are disjoint, and thus the events concerning the validity of different swaps are independent.
Hence, the probability that Assumption~\ref{ass4} fails at step $i$ is at most
\[
(1-\eta^2/2)^{ |M_i| |M'_{\pi(i)}| } \le (1-\eta^2/2)^{ (\e^6/8+o(1))\log^2n} = o(1/N_1).
\]
Summing over all $N_1$ potential steps, we conclude that a.a.s.\ Assumption~\ref{ass4} holds throughout the procedure.
In view of that, we pick a valid $xy,uv$-swap arbitrarily, apply it to cycles $H_{i-1}$ and $D_i$, and call $H_i$ the resulting cycle.
After the swap, we update the matchings as follows.  We delete edge $uv$ from $M'_{\pi(i)}$ and also from $M_{\pi(i)}$. Moreover, we delete $xy$ from $M_i$, and partition the resulting matching $M_i$ into two disjoint matchings $M'_i$ and $M''_i$ of size at least $(\e^3/4+o(1))\log n$ each.
This finalizes step $i$. We now verify that Claims~\ref{claim1}--\ref{claim5} remain valid at the end of this step.
By construction, $H_i$ satisfies all the properties in Claim~\ref{claim1}.
Claim~\ref{claim2} is also true since we did not expose the colors on any edge incident with any vertex in $V_j$ for $j>i$.
Moreover, since the only edges in $H_i\setminus H_{i-1}$ with endpoints in different cells are $xu$ and $yv$, then
\[
|H_i\cap E_{j,j'}| = \begin{cases}
|H_{i-1}\cap E_{j,j'}| & \text{for $1\le j'<j\le i-1$}
\\
2 & \text{for $j=i$ and $j'=\pi(i)$}
\\
0 & \text{for $j=i$ and $j'\ne\pi(i)$,}
\end{cases}
\]
which implies that Claim~\ref{claim3} remains valid.
Claim~\ref{claim4} still holds since we did not label any new cell unsafe.
Matchings $M_i,M'_i,M''_i$ introduced at this step satisfy the properties in Claim~\ref{claim5} by construction. Note that we did not expose the colors of any edges in $E_{i,\pi(i)}$ incident with $M''_{\pi(i)}$.
Matchings $M_j,M'_j,M''_j$ with $j<i$ satisfied Claim~\ref{claim5} at the previous step, but we must take into account that the sizes of matchings $M_{\pi(i)},M'_{\pi(i)}$ were decreased by one. However, each matching can only be affected at most $\Delta(T)=O_\e(1)$ times througout the procedure as a result of a {\tt \xt{C}ycle-patch} step, so Claim~\ref{claim5} holds.

\paragraph{{\tt Forest-patch} step:}
Otherwise, suppose $G_i$ is not Hamiltonian. In that case, we label cell $C_i$ as unsafe.
A {\em linear forest} is a graph whose connected components are paths.
By Assumption~\ref{ass3}, we can pick a spanning linear forest $L_i$ of $G_i$ with at most $\psi_0$ components.
Note that $H_{i-1} \cup L_i$ is rainbow by construction and does not use any of the colors used on $E'$.
Since $L_i$ consists of at most $\psi_0$ paths with a total of at least $m_i-\psi_0 \ge(\e^3+o(1))\log n$ edges, we can choose a matching $M_i$ of size at least $(\e^3/2+o(1))\log n$ contained in $L_i$, and then partition $M_i$ into two disjoint matchings $M'_i,M''_i$ of size at least $(\e^3/4+o(1))\log n$ in any arbitrary way.
Moreover, recall that $M''_{\pi(i)} \subseteq E_{\pi(i)} \cap H_{i-1}$ is a matching of size at least
$(\e^3/4+o(1))\log n$, by Claim~\ref{claim5} applied to step $i-1$.

Our goal is to patch each of the path components of $L_i$ into $H_{i-1}$.
For each path component $P$ of $L_i$, we reveal the colors on all the edges in $E_{i,\pi(i)}$ that are incident with both an endpoint of $P$ and some edge in $M''_{\pi(i)}$. These colors had not been exposed yet in view of Claim~\ref{claim2} at step $i-1$.
Then, we apply a {\tt \xt{P}ath-patch} sub-step (below) to this path $P$. \xt{If successful, this sub-step} extends cycle $H_{i-1}$ to a larger rainbow cycle that contains $P$ and does not use any colors previously used on $E'$. For convenience, we still call this new cycle $H_{i-1}$, but will rename it to $H_i$ at the end of the step when all the paths of $L_i$ have been inserted.

\paragraph{{\tt Path-patch} sub-step:}
Let $u,v$ be the endpoints of path $P$ in cell $C_i$ (possibly $u=v$).
We can assume that path $P$ is directed, say from $v$ to $u$, by appropriately modifying the orientation of the edges if necessary.
Our goal is to patch $P$ into $H_{i-1}$. To do that, we will pick an appropriate edge $xy\in M''_{\pi(i)}$, delete $xy$ from $H_{i-1}$, and add edges $xv,uy$ to join directed paths $H_{i-1}-xy$ and $P$.
(As usual, we update the orientations in $G_{\cX,\hat r}$ of the two added edges $xv,uy$, if needed.)
By analogy with the {\tt \xt{C}ycle-patch} step, we can regard this operation as performing an $xy,uv$-swap to $H_{i-1} \cup (P+uv)$,
where $P+uv$ denotes the directed cycle obtained by adding edge $uv$ to path $P$.
(If $u=v$, then $P+uv$ is simply a loop $uu$; if $P$ consists of one single edge $vu$, then we simply regard $P+uv$ as a pair of anti-parallel edges.)
Note that the way $P$ is inserted into $H_{i-1}$ depends on the orientation given to $P$. (For simplicity, our procedure only considers one of the two possible ways of doing that.)

Given an edge $xy\in M''_{\pi(i)}$, let $c_1$ and $c_2$ denote the colors assigned in $G_{\cX,\hat r, q}$ to the edges that would be added at the end of an $xy,uv$-swap (i.e.~$xv$ and $uy$).
If color $c_k$ is repeated on an edge $e_k\in H_{i-1}$ for some $k\in\{1,2\}$, then we say that the $xy,uv$-swap causes a {\em problem} at edge $e_k$ or simply that $e_k$ is a {\em problem edge} (relative to that particular swap and $H_{i-1}$).
Note that each color $c_k$ appears on at most one edge of $H_{i-1}$ (since $H_{i-1}$ is rainbow), and therefore an $xy,uv$-swap causes at most two problems.
We say that the $xy,uv$-swap is {\em ideal} if colors $c_1$ and $c_2$ are different from each other and do not appear on any edges in $E' \cup H_{i-1} \cup E_i$.
(In particular, ideal swaps create no problem edges.)
On the other hand, the $xy,uv$-swap is {\em acceptable} if it is not ideal but the following conditions hold:
1) colors $c_1$ and $c_2$ are different;
2) $c_1$ and $c_2$ do not appear on any edge in $E'\cup E_i$;
3) each $c_k$ is used at most once on $H_{i-1}$ (note that this condition is redundant since $H_{i-1}$ is rainbow, but it will be useful later on when we consider acceptable swaps in a slightly different context that allows a few color repetitions);
4) if color $c_k$ is used on some edge $e_k \in H_{i-1}$ for some $k\in\{1,2\}$ (i.e.~$e_k$ is a problem edge), then $e_k \in E_{j_k}$ for some safe cell $C_{j_k}$; and
5) if both colors $c_1,c_2$ are respectively used on edges $e_1,e_2\in H_{i-1}$, then the safe cells $C_{j_1}$ and $C_{j_2}$ containing these edges (as defined in condition 4)) must be different.
Later on, we will consider acceptable swaps in a context where $H_{i-1}$ may already contain some additional edges labelled as problems, which are located at different safe cells. In view of that, it is convenient to reword condition 5) as follows:
a problem edge created by the $xy,uv$-swap cannot be contained in the same cell as another problem edge (relative to that swap or already present in $H_{i-1}$).
Finally, the $xy,uv$-swap is {\em forbidden} if it is neither ideal nor acceptable.

\begin{ass}\label{ass5}
Not all the $xy,uv$-swaps for $xy\in M''_{\pi(i)}$ are forbidden.
(Hence there is at least one swap that is ideal or acceptable.)
\end{ass}
We defer the proof that Assumption~\ref{ass5} is valid a.a.s., until later.

First suppose that there exists an edge $xy\in M''_{\pi(i)}$ such that the $xy,uv$-swap is ideal. Pick one such edge $xy$ arbitrarily, and apply the $xy,uv$-swap to $H_{i-1}\cup(P+uv)$. This inserts $P$ into $H_{i-1}$. The resulting cycle, which we still call $H_{i-1}$, is directed and rainbow by construction and does not contain any colors used on $E'$. After performing the swap, edge $xy$ is removed from matching $M''_{\pi(i)}$ and thus from $M_{\pi(i)}$.
Otherwise, if there is no ideal swap available, pick an arbitrary $xy\in M''_{\pi(i)}$ such that the $xy,uv$-swap is acceptable (there must be at least one \xt{by Assumption~\ref{ass5}}).
By definition, for at least one $k\in\{1,2\}$ (and maybe for both), color $c_k$ already appears on one edge $e_k\in H_{i-1}$ which is labelled as a problem edge.
We then apply the $xy,uv$-swap to $H_{i-1}\cup(P+uv)$, and remove edge $xy$ from the matchings $M''_{\pi(i)}$ and $M_{\pi(i)}$.
In this case, the new cycle obtained after the swap, which we still denote by $H_{i-1}$, is directed and contains no colors used on $E'$, but it is not rainbow since it has one or two color repetitions, one per problem edge.
We now attempt to make $H_{i-1}$ rainbow by taking a {\tt Problem-fix} sub-step (below) for each problem edge $e_k$. This procedure recursively applies cycle rotations to $H_{i-1}$, \xt{each one of which removes one problem edge but may in turn create at most two new problem edges.}
\begin{ass}\label{ass6}
{\tt Problem-fix} successfully terminates after at most $\xi= \xt{ \lceil 17 / (\eta^2 \e^3) \rceil }$ recursive iterations.
\end{ass}
We defer the proof that Assumption~\ref{ass6} is valid a.a.s., until later.
\xt{Note that in view of this assumption, $H_{i-1}$ never contains more than $\xi$ problem edges, since each problem edge triggers an iteration of {\tt Problem-fix}.}
We will show that, \xt{after all the recursive iterations of {\tt Problem-fix} in the 
{\tt Path-patch} sub-step successfully terminate,} the resulting cycle $H_{i-1}$ is directed and rainbow, includes path $P$ and does not share any colors with $E'$.
This ends the {\tt Path-patch} sub-step. If $P$ was the last path of $L_i$ to be inserted, then rename $H_{i-1}$ to $H_i$.

\paragraph{{\tt Problem-fix} sub-step:}
Recall that $H_{i-1}$ is a directed cycle, but not rainbow. However, all its color repetitions are due to the presence of problem edges.
\xt{Also recall that $H_{i-1}$ may contain up to $\xi$ problem edges in view of Assumption~\ref{ass6}.}
Moreover, since problem edges originate from acceptable swaps, they must all lie in different and safe cells by construction.

Suppose that we are trying to fix a problem edge $uv\in H_{i-1}\cap E_j$ for some $j\le i-1$. In particular, the cell $C_j$ containing that edge must be safe, and thus by Claim~\ref{claim5} the colors on the edges that are incident with $uv$ and $M''_{\pi(j)}$ have not yet been exposed.
Our plan is to perform an $xy,uv$-rotation of $H_{i-1}$ for some suitable $xy\in M''_{\pi(j)}$.
(Note that $uv$ and $xy$ are not incident, since they are contained in different cells.)
This operation amounts to reversing the orientation of some edges (including $uv$ to $vu$) and applying an $xy,vu$-swap. For simplicity, we will only discuss the choice of the $xy,vu$-swap for $xy\in M''_{\pi(j)}$, and assume that the edge orientations are adjusted as required by the $xy,uv$-rotation, so the resulting cycle (which we still call $H_{i-1}$) is directed.

We essentially follow the same strategy as in the {\tt Path-patch} sub-step.
We reiterate Assumption~\ref{ass5} here, and suppose that not all the $xy,vu$-swaps are forbidden. (Otherwise, {\tt Build} fails.)
Then, we first attempt to perform an ideal $xy,vu$-swap for some $xy\in M''_{\pi(j)}$, if possible, and otherwise use an acceptable one.
In the former case, the corresponding $xy,uv$-rotation successfully removes the problem edge $uv$ from $H_{i-1}$ while not creating any new problems. In the latter case, we also get rid of $uv$, but add one or two new problem edges (and thus color repetitions) to $H_{i-1}$.
At the end of either case, we further delete edge $xy$ from the matchings $M''_{\pi(j)},M_{\pi(j)}$, and remove edge $uv$ from any of the matchings $M_{j},M'_{j},M''_{j}$ that may contain it (possibly none of them).
Finally, we update the status of cell $C_j$ from safe to unsafe to guarantee that it will never host other problem edges at any other step of the algorithm.
If the resulting directed cycle $H_{i-1}$ has no problem edges left, then we successfully terminate {\tt Problem-fix}. Otherwise, we recursively apply {\tt Problem-fix} to one of the remaining problem edges.
With Assumption~\ref{ass6} in mind, we only allow up to $\xi$ recursive iterations arising from one {\tt Path-patch} sub-step. Otherwise, {\tt Build} fails.

This ends the description of the {\tt Forest-patch} step (and all its corresponding sub-steps), which is taken at step $i$ if $G_i$ is not Hamiltonian. We proceed to verify that at the end of that step Claims~\ref{claim1}--\ref{claim5} remain valid.
Claim~\ref{claim1} holds by construction, inductively assuming that it was true at step~$i-1$. Indeed, after inserting all paths from $L_i$ into $H_{i-1}$ and recursively fixing all the problem edges, $H_i$ is a directed rainbow cycle spanning $V_1\cup\cdots\cup V_i$ and avoiding all colors that appear on $E'$.
Claim~\ref{claim2} also remains valid since only colors on edges in $E_i\cup \bigcup_{1<j\le i} E_{j,\pi(j)}$ were exposed at step $i$.
To verify Claim~\ref{claim3}, note that $|H_i\cap E_{j,j'}| = |H_{i-1}\cap E_{j,j'}|$ for each $1\le j'<j\le i-1$, unless $j'=\pi(j)$ and a problem edge was created in $C_j$ during step $i$, in which case $|H_i\cap E_{j,\pi(j)}| = |H_{i-1}\cap E_{j,\pi(j)}| + 2$.
Moreover, $|H_i\cap E_{i,\pi(i)}| \le 2\psi_0$ and $|H_i\cap E_{i,j}| = 0$ for $j\ne\pi(i)$.
Hence, Claim~\ref{claim3} follows by induction and from the fact that a cell can host at most one problem edge during the whole procedure.
%
%
Now recall that $C_1$ is always unsafe and that a cell $C_j$ ($1<j\le i$) is unsafe at the end of step $i$ only in the following two situations: a) step $j$ was not Hamiltonian (there are at most $n^{1-\e^3\eta/\xt{11}}$ such steps, by Assumption~\ref{ass2}); or b) step $j$ was Hamiltonian (and thus $C_j$ was initially declared safe), but then $C_j$ became unsafe due to a problem edge arising from a later non-Hamiltonian step (there \xt{are} at most $n^{1-\e^3\eta/\xt{11}}$ non-Hamiltonian steps, by Assumption~\ref{ass2}, and each triggers at most $\xi$ problem edges, by Assumption~\ref{ass6}).
Hence, there are at most $1 + n^{1-\e^3\eta/\xt{11}} + \xi n^{1-\e^3\eta/\xt{11}} = o(N_1)$ unsafe cells, and Claim~\ref{claim4} holds.
Next we verify Claim~\ref{claim5} at the end of step~$i$, inductively assuming that it was true at the previous step.
Note that matchings $M_i,M'_i,M''_i$ created in the {\tt Forest-patch} step satisfy all the requirements by construction (recall that $C_i$ is declared unsafe, so the last condition in the claim is trivially true).
We need to check that $M_j,M'_j,M''_j$ (for $1\le j\le i-1$) still meet all the conditions at the end of step $i$.
During that step, the sizes of matchings $M_{\pi(i)}$ and $M''_{\pi(i)}$ decreased by at most $\psi_0$ due to the insertion of the paths of $L_i$ into $H_{i-1}$, but this can happen at most $\Delta(T)=O_\e(1)$ times throughout the entire procedure.
Moreover, for each cell $C_j$ ($1<j\le i-1$) containing a problem edge at step $i$, we decreased the sizes of $M_j,M'_j, M''_j$ by at most one (but this can happen only once in the procedure), and likewise the sizes of $M_{\pi(j)},M''_{\pi(j)}$ were decreased by one (but this can happen at most $\Delta(T)=O_\e(1)$ times). We excluded $j=1$ above since $C_1$ is unsafe, and thus never contains problem edges. Hence, by induction, $|M'_j|, |M''_j| \ge (\e^3/4+o(1)) \log n$ and $|M_j| \ge (\e^3/2+o(1)) \log n$ for all $1\le j\le i$.
Finally, all the cells $C_j$ ($1<j\le i-1$) for which we exposed the colors on the edges in $E_{j,\pi(j)}$ incident with $M''_{\pi(j)}$ at step $i$ were relabelled unsafe, and therefore the last condition in Claim~\ref{claim5} remains valid.

We have shown that Claims~\ref{claim1}--\ref{claim5} hold throughout the $N_1$ steps of procedure {\tt Build} as long as it does not fail: that is, under Assumptions~\ref{ass1}--\ref{ass6}. In particular, at the end of the $N_1$-th step, by Claim~\ref{claim1}, we obtain a rainbow cycle $H_{N_1}$ spanning all the vertices in rainbow good cells and avoiding all the colors used on $E'$.
It only remains to show that Assumptions~\ref{ass5} and~\ref{ass6} hold a.a.s.\ through all the steps (since Assumptions~\ref{ass1}--\ref{ass4} have already been verified) in order to conclude that procedure {\tt Build} succeeds a.a.s.
To do so, we will bound the probability that a given swap in a {\tt Path-\xt{patch}} or a {\tt Problem-fix} sub-step is forbidden and the probability it is not ideal.

\xt{Let $1<i\le N_1$, and suppose that step $i$ is non-Hamiltonian (i.e.~the procedure performs a {\tt Forest-patch} step).}
Let $u,v\in V_j$ be the endpoints of the problem edge to be fixed in a {\tt Problem-fix} \xt{sub-step within that step} with $j\le i-1$ or the endpoints of the path to be patched in a {\tt Path-patch} sub-step with $j=i$.
Given $xy\in M''_{\pi(j)}$, the $xy,uv$-swap is forbidden if, for some $k\in\{1,2\}$, color $c_k$ is equal to $c_{3-k}$ or one of the repeated colors on $H_{\xt{i-1}}$ (there are at most $\xi$ of those, by Assumption~\ref{ass6}, since each color repetition is due to a problem edge) or if $c_k$ appears on any of the following edges:
\xt{1)}
edges in $E'\cup E_i$ (where $|E'\cup E_i| \le n^{1-\e/3} + \xt{\binom{\log n}{2}}$, by an earlier bound on $|E'|$ and {\bf P1}),
\xt{2)}
edges in \xt{$H_{i-1}$ contained in} unsafe cells (there are $o(N_1 \log n)$ of those \xt{edges}, by Claim~\ref{claim4} and \xt{the fact that each unsafe cell contains at most $\log n$ edges of $H_{i-1}$ --- by~{\bf P1} and since $H_{i-1}$ is a cycle}),
\xt{3)}
edges in $H_{\xt{i-1}}$ with endpoints in different cells (there are $\xt{O_\e}(N_1)$ of those, by Claim~\ref{claim3})
or \xt{4)}
edges in \xt{$H_{\xt{i-1}}$ and inside} a safe cell containing another problem edge (there are at most $\xi \log n$ of these, by Assumption~\ref{ass6}, \xt{{\bf P1} and the fact that $H_{\xt{i-1}}$ is a cycle}).
Hence, the probability that a given $xy,uv$-swap is forbidden is at most
\[
2\left( \frac{1 + \xi + |E'| + \xt{\binom{\log n}{2}} + o(N_1\log n) + \xt{O_\e}(N_1) + \xt{O_\e}(\log n)}{|Q|} \right) = o(1).
\]
Since $M''_{\pi(j)}$ is a matching, events concerning different swaps are independent, and thus the probability that all the swaps are forbidden is
\[
(o(1))^{|M''_{\pi(j)}|} \le  (o(1))^{(\e^3/4+o(1))\log n} = n^{-\omega(1)} = o(1/N_1),
\]
by Claim~\ref{claim5}.
Therefore, summing this bound over all $N_1$ potential steps times the at most $(1+\xi)$ possible {\tt Path-patch} or {\tt Problem-fix} sub-steps within each step, the probability that Assumption~\ref{ass5} fails at some point in the algorithm is $o(1)$.
%
\xt{On the other hand, given $xy\in M''_{\pi(j)}$, recall that the $xy,uv$-swap is ideal if colors $c_1$ and $c_2$ are different from each other and do not appear on any edges in $E'\cup H_{i-1}\cup E_i$.
Hence, the probability that a given $xy,uv$-swap is not ideal is at most
\[
1 - \left( 1 - \frac{1 + |E'| + n + \binom{m_i}{2}}{|Q|} \right)^2 = 1 - \left(\frac{\eta+o(1)}{1+\eta}\right)^2 \le 1 - \eta^2/2
\]
(for $0<\eta<\sqrt2 - 1$ and large enough $n$).}
Thus, the probability that we are forced to pick an acceptable swap at a given {\tt Path-patch} or {\tt Problem-fix} sub-step is at most
\[
(1-\xt{\eta^2}/2)^{|M''_{\pi(j)}|} \le (1-\xt{\eta^2}/2)^{(\e^3/4+o(1))\log n} \le n^{-\xt{\eta^2}\e^3/8+o(1)},
\]
again by Claim~\ref{claim5}.
Note that each acceptable swap introduces one or two new problem edges, which in turn require recursive iterations of {\tt Problem-fix}.
Then, the probability that from one single {\tt Forest-patch} step we create $\xi$ problems (which requires picking at least $\xt{ \lceil\xi/2\rceil }$ acceptable swaps) is at most
\[
O( n^{- \xi \xt{\eta^2} \e^3/16+o(1)} ) = o(1/N_1),
\]
where we use the fact that $\xi \xt{\eta^2} \e^3/16 > 1$.
So we expect $o(1)$ violations of Assumption~\ref{ass6} in the $N_1$ steps of the algorithm, and thus Assumption~\ref{ass6} holds a.a.s.\  by Markov inequality.
This completes the analysis of {\tt Build}, and shows that a.a.s.\ we obtain a rainbow directed cycle $H=H_{N_1}$ through all the vertices inside rainbow good cells that avoids all the colors used on $E'$.


\subsection{Bad or non-rainbow cells}
Suppose that all the earlier a.a.s.\ statements in the paper hold (see the discussion at the beginning of Section~\ref{sec:coloring}) and also that procedure {\tt Build} succeeds at building the rainbow directed cycle $H$. (Here we assume that the edges of $G_{\cX,\hat r}$ are oriented as in Section~\ref{ssec:good}.)
Recall that $H$ does not use any colors on $E'$, which is the set of edges of $G_{\cX,\hat r}$ that are incident with a point in a cell that is not good or not rainbow or are incident with a point in an ugly path.
We will extend $H$ by adding the points in cells that are either non-rainbow
(and thus good by Lemma~\ref{lemcols}\xt{(e)})
or bad, one cell at a time. We will do that deterministically, given all our a.a.s.\ assumptions.

Let $C$ be a cell that is bad or non-rainbow. By the definition of bad cell and by Lemma~\ref{lemcols}\xt{(e)}, $C$ must be adjacent in the graph of cells $\cG_\cC$ to some rainbow good cell $C_i$ (for some $1\le i\le N_1$). In view of Claim~\ref{claim5}, at the end of procedure {\tt Build}, $H\cap E_i$ contains a matching $M_i$ of size $|M_i|=\Omega_\e(\log n)$.
An edge in $E'(C)$ (i.e.~incident with some point in $V(C)$) is labelled {\em dangerous} if its color in $G_{\cX,\hat r,q}$ is repeated on some other edge in $E'$. By Lemma~\ref{k0} there can be at most $k_0+1=O_\e(1)$ dangerous edges in $E'(C)$.
Suppose first that $V(C)$ contains more than $2k_0+5$ points.
Since $V(C)$ induces a clique in $G_{\cX,\hat r}$, we can find $k_0+2$ edge-disjoint spanning cycles of that clique (for instance, consider the well-known Walecki construction described in~\cite{Alspach}). At least one of these cycles does not contain any dangerous edges. Pick one and call it $H_C$. We can assume that $H_C$ is a directed cycle by adjusting the orientations of its edges as needed.
Now pick an edge $uv \in H_C$ and an edge $xy \in M_i$ with the property that $xy$ is not incident with any dangerous edge in $E(C')$. (We have at least $|M_i|-k_0-1 = \Omega_\e(\log n)$ choices for $xy$, since a dangerous edge in $E(C')$ is incident with at most one edge in $M_i$.) Then, by applying an $xy,uv$-swap to $H_C \cup H$, we merge $H_C$ and $H$ into one larger rainbow directed cycle that we still call $H$. After the swap, delete $xy$ from $M_i$.
Otherwise, if $V(C)$ contains $t \le 2k_0+5$ points $v_1,\ldots,v_t$, then pick $t$ different edges $x_1y_1,\ldots,x_ty_t$ in $M_i$ such that each $x_jy_j$ is not incident with any dangerous edge in $E'(C)$. As before, we have plenty of freedom to do this, since we can choose from a pool of at least $|M_i|-k_0-1 = \Omega_\e(\log n)$ edges. Then, each vertex $v_j$ ($1\le j\le t$) is inserted into $H$ by replacing $x_jy_j$ by the directed path $x_jv_jy_j$, ajusting edge orientations if needed.
The resulting cycle, which we still denote by $H$, is directed and rainbow by construction and includes all the points in $V(C)$. After doing that, we delete edges $x_1y_1,\ldots,x_ty_t$ from the matching $M_i$.

We repeat the same operation for every bad or non-rainbow cell $C$, one cell at a time, until $H$ covers all vertices of $\cX$ that are not in ugly paths.
Note that, since the graph of cells has maximum degree $\Delta(\cG_\cC) = O_\e(1)$ \xt{(from~\eqref{eq:degree})}, for each good rainbow cell $C_i$ ($1\le i\le N_1$), the corresponding matching $M_i$ may loose at most $(2k_0+5) \Delta(\cG_\cC) = O_\e(1)$ edges in total, so we still have $|M_i| = \xt{\Omega_\e}(\log n)$ throughout this procedure, as required.
Moreover, the edges that were used to extend $H$ must have all different colors and do not repeat colors from edges incident with ugly paths, thanks to the fact that we did not choose any dangerous edges.
Therefore, we eventually obtain a rainbow directed cycle $H$ that covers all vertices in $\cX\setminus V(\cP)$ and avoids all colors on $E'(\cP)$.


\subsection{Ugly paths}
It only remains to patch the ugly paths into $H$. Let $\cP$ be the collection of ugly paths as in Lemma \ref{lem1}.
For each path $P\in\cP$, let $u$ and $v$ be the endpoints of $P$, and assume that $P$ is directed from $v$ to $u$ by adjusting the orientations of its edges appropriately if necessary.
Let $C_P$ be a good cell as in~{\bf Q4}, which must be also rainbow by Lemma~\ref{lemcols}\xt{(e) and~{\bf Q3}}. So, using the notation from Section~\ref{ssec:good}, $C_P=C_i$ for some $1\le i\le N_1$. Pick any edge $xy$ from the matching $M_i$. (There are $|M_i|=\Omega_\e(\log n)$ choices.) Note that both $x$ and $y$ are adjacent with $u$ and $v$ in $G_{\cX,\hat r}$ by our choice of $C_P$. By applying an $xy,uv$-swap to $H\cup (P+uv)$, we insert $P$ into $H$. The resulting cycle, which we still call $H$, is directed and rainbow, since the set of edges $E'(\cP)$ incident with ugly paths is rainbow by Lemma~\ref{rainP} and $H$ does not use any colors appearing on $E'(\cP)$.
We can repeat this operation for each $P\in\cP$, noting that each rainbow good cell $C_i$ will be used to patch at most one ugly path in view of~{\bf Q5}. Hence, we eventually obtain a rainbow (directed) Hamilton cycle $H$ of $G_{\cX,\hat r,q}$.
This completes the proof of Theorem~\ref{th2}.

\section*{Acknowledgements}
The authors wish to thank an anonymous referee for their careful review and helpful comments that significantly improved the quality of the manuscript.

\end{document}